\documentclass[11pt]{article}
\usepackage[a4paper, total={6in, 8in}]{geometry}
\usepackage{enumerate}
\usepackage{amsmath}
\usepackage{amssymb,latexsym}
\usepackage{amsthm}
\usepackage{color}
\usepackage{cancel}
\usepackage{graphicx}
\usepackage{cite}
\usepackage{changes}
\usepackage{lineno}
\usepackage{hyperref}
\usepackage{comment}
\usepackage{amscd}

\newtheorem{theorem}{Theorem}[section]

\newtheorem{lemma}[theorem]{Lemma}
\newtheorem{corollary}[theorem]{Corollary}
\newtheorem{definition}[theorem]{Definition}
\newtheorem{proposition}[theorem]{Proposition}

\newtheorem{example}[theorem]{Example}

\newtheorem{remark}[theorem]{Remark}

\newtheorem{open}[theorem]{Open Problem}

\newcommand{\fqn}{\mathbb{F}_{q^n}}

\newcommand{\F}{{\mathbb F}}

\newcommand{\Tr}{{\mathrm Tr}}

\newcommand{\fq}{{\mathbb F}_{q}}

\newcommand{\la}{\langle}
\newcommand{\ra}{\rangle}

\newcommand{\PG}{\mathrm{PG}}
\newcommand{\N}{\mathrm{N}}

\title{On the weight distribution of linear sets with complementary weights and related constructions}
\author{Geertrui Van de Voorde \thanks{University of Canterbury, Chrsitchurch, New Zealand} \and Ferdinando Zullo \thanks{Universit\`a degli studi della Campania Luigi Vanvitelli, Caserta, Italy}}
\date{ }

\begin{document}
\maketitle

\begin{abstract}
In this paper, we continue the study of linear sets with complementary weights. We find criteria to determine the set of points of any fixed weight and use this to present particular linear sets with few points of weight more than one. 

We also present a product-type construction for linear sets of complementary type arising from any linear set, allowing us to control the weight distribution of the obtained linear set. Finally, we use this construction to create linear sets with many different weights, along with point sets of even type with many distinct intersection numbers.
\end{abstract}


\section{Preliminaries}
In the past decade, linear sets have become one of the central objects of study within finite geometry. Originally motivated by their connections with linear blocking sets and translation ovoids \cite{translationovoids}, they have increased importance since the connection with rank-metric codes was established \cite{sheekey2016new}. We refer to \cite{polverino2010linear,lavrauw2015field} for more background and applications.

An $\F_q$-linear set $L_U$ in $\PG(1,q^n)$ of rank $k$ is a point set whose defining vectors form a $k$-dimensional $\F_q$-vector space $U$. More formally, writing $\langle u\rangle_{\F_{q^n}}$ for the projective point in $\PG(1,q^n)$ obtained from all $\F_{q^n}$-multiples of a vector $u\in \F_{q^n}^2$, we have that $$L_U=\{\langle u\rangle_{\F_{q^n}}\mid u\in U\setminus\{0\}\}. $$
The weight of a point $\langle u\rangle_{\F_{q^n}}$ in $L_U$ is defined as the $\F_q$-dimension of the subspace $\langle u\rangle_{\F_{q^n}}\cap U$. 

A linear set of rank $n$ in $\PG(1,q^n)$ can be described via an $\F_q$-linearised polynomial as follows. Let $\mathcal{L}_{n,q}$ be the set of all $\F_q$-linearised polynomials over $\F_{q^n}$, that is, all polynomials of the form $f(x)=\sum_{i=0}^{n-1}a_i x^{q^i},$ where $a_i\in \F_{q^n}$. Then it is well-known that an $\F_q$-linear set of rank $n$ in $\PG(1,q^n)$, disjoint from the point $\langle(0,1)\rangle_{\F_{q^n}}$ can always be written as $\{\langle(x,f(x)) \rangle_{\F_{q^n}}\mid x\in \F_{q^n}^*\}$ for some $f\in \mathcal{L}_{n,q}.$

In this paper, we will focus on a particular class of linear sets, namely those with complementary weights. 
An $\fq$-linear set $L_U$ of rank $k$ in $\PG(1,q^n)$ is a set with {\em complementary weights} if it has two points $P_1,P_2$ whose weights sum to $k$.

Linear sets with complementary weights have been studied before, see e.g. \cite{napolitano2022linear,napolitano2022classifications,jena2021linear} and \cite{adriaensen2024minimum,santonastaso2025completely,zullo2023multi} for some applications.
We will list some of their main properties in the next Section.
In this note, we will contribute to this study as follows:
\begin{itemize}
\item We give a new criterion to determine the set of points of any fixed weight in a linear set of complementary type (see Theorem \ref{thm:weightSxT}, and Corollary \ref{cor:weightscomplpoli} for an algebraic description).
\item We use these criteria to present linear sets for which we can completely determine the weight distribution (see Propositions \ref{prop:tracextrace} and \ref{prop:powerqxpowerq}) and for which we show that only a few weights are possible (Proposition \ref{lem:tracewithxq} and Proposition \ref{prop:LPxLP}).
\item We present a product-type construction for linear sets of complementary weight in $\PG(1,q^{2t})$ from any linear set in $\PG(1,q^t),$ and show how their weight distributions are related (see Theorem \ref{thm:bigfromsmall}).
\item We use this product construction to create examples of linear sets with {\em many} different weights (see Corollary \ref{cor:manyweights}), and even sets in $\mathrm{PG}(2,q)$ with many different intersection numbers (see Corollary \ref{cor:evenset}).

\end{itemize}

\section{Introduction and previous results}

In the paper~\cite{napolitano2022linear}, the authors considered linear sets of rank~$k$ in $\PG(1, q^n)$ with complementary weights~$r$ and~$k - r$. 

First, we observe that, by~\cite[Proposition~3.2]{napolitano2022linear}, if $L_W$ is an $\F_q$-linear set of rank~$k$ in $\PG(1, q^n)$ for which there exist two distinct points $P, Q \in L_W$ such that 
\[
w_{L_W}(P) = r \quad \text{and} \quad w_{L_W}(Q) = k - r,
\]
then $L_W$ is $\mathrm{PGL}(2, q^n)$-equivalent to a linear set $L_U$, where $U = S \times T$ for some $\F_q$-subspaces $S, T \subseteq \F_{q^n}$ of dimensions~$r$ and~$k - r$, respectively. 
Therefore, whenever we consider linear sets with complementary weights, we assume that they are defined by a subspace of the form $S \times T$, as above.

Also, in \cite{napolitano2022classifications}, it was derived a way to find the points with second possible maximum weight (when $r\neq k-r$) or the maximum weight (when $r=k-r$). More precisely, they showed the following:

\begin{theorem}\label{thm:numberopointsr}\cite[Theorem 3.4]{napolitano2022classifications}
Let $U=S\times T$, where $S$ is a $(k-r)$-dimensional $\fq$-subspace of $\fqn$ with $r \leq k-r$ and $T$ is an $r$-dimensional $\F_q$-subspace of $\F_{q^n}$ and suppose that $T=\langle a_1,\ldots,a_r\rangle_{\fq}$ for some $a_1,\ldots,a_r \in \fqn$.
The set of points of weight $r$ in $L_U$ different from $\langle (1,0)  \rangle_{\fqn}$ is 
\[ \{ \langle (\xi,1)\rangle_{\fqn} \colon \xi \in a_1^{-1}S\cap \ldots \cap a_r^{-1} S \} \]
and its size is $q^j$ with $j=\dim_{\fq}(a_1^{-1}S\cap \ldots \cap a_r^{-1} S)$.
\end{theorem}

Recall that the linear set $L_{S\times T}$ as defined in Theorem \ref{thm:numberopointsr} contains the point $\langle (1,0)\rangle_{\F_{q^n}}$ which has weight $k-r\geq r$. Therefore, all other points have weight at most $r$, and we know that $\langle(0,1)\rangle_{\F_{q^n}}$ has weight $r$.
We will first use the condition derived in Theorem \ref{thm:numberopointsr} to show that it is always possible to obtain a linear set $L_{S\times T}$ with, apart from $\langle(0,1)\rangle_{\F_{q^n}}$, no further points of weight $r$:
\begin{proposition} For any $(k-r)$-dimensional $\fq$-subspace $S$ in $\fqn$ with $k-r \leq n/2$, $k \leq n$, $r\geq 2$, there exists an $r$-dimensional $\fq$-subspace $T$ in $\fqn$ such that the $\fq$-linear set $L_{S\times T}$ of rank $k$ in $\PG(1,q^n)$ has no points of weight $r$ apart from  $\langle(0,1)\rangle_{\F_{q^n}}$ and $\langle (1,0)\rangle_{\F_{q^n}}$ (when $k-r=r$).
\end{proposition}
\begin{proof}
We first show that we can find an element $\xi \in \fqn$ such that $S\cap \xi S=\{0\}$. Suppose to the contrary that for any $\xi \in \fqn$ we have that $\dim_{\fq}(S\cap \xi S)\geq 1$, so for all $\xi$, there exists some $s\in S\cap \xi S$. This implies that $s=\xi s'$ for some $s'\in S$, and hence, that $\langle (1,\xi)\rangle_{\F_{q^n}}=\langle (s',\xi s')\rangle_{\F_{q^n}}=\langle (s',s)\rangle_{\F_{q^n}}\in L_{S\times S}.$
Since $\langle (0,1)\rangle_{\F_{q^n}}\in L_{S\times S}$ too, it follows that $L_{S\times S}=\PG(1,q^n)$. But since the rank of $L_{S\times S}$ is $2(k-r)\leq n$, this is a contradiction. Therefore, there exists $\xi \in \fqn$ such that $S\cap \xi S$ has dimension zero.
Let $T$ be an $\fq$-subspace of $\fqn$ of dimension $r$ containing $1$ and $\xi$ and suppose that $\{1,\xi,a_3,\ldots,a_{r}\}$ is its basis.
Since 
\[ S\cap \xi^{-1}S\cap a_3^{-1}S \cap \ldots a_{r}^{-1}S\subseteq S\cap \xi^{-1}S=\{0\}, \]
by Theorem \ref{thm:numberopointsr} we get the assertion.
\end{proof}

In the paper \cite{napolitano2022linear}, the authors investigate the possibility of the linear set $L_{S\times T}$ having no points with weight more than one, apart from $\langle(1,0)\rangle_{\F_{q^n}}$ and $\langle(0,1)\rangle_{\F_{q^n}}$; the following corollary explains the interest in linear sets of rank $2t$ with exactly two points of weight $t$ in $\PG(1,q^{2t})$. 

\begin{corollary}\label{cor:UasUST}\cite[Corollary 4.3]{napolitano2022linear}
Let $L_U$ be an $\fq$-linear set of rank $k$ in $\PG(1,q^n)$ for which there exist two distinct points $P,Q \in L_U$ such that
$w_{L_U}(P)=s$, $w_{L_U}(Q)=t$, $s+t=k\leq n$ and $s,t>1$. 
If $P$ and $Q$ are the only points of $L_U$ of weight greater than one, then $s\leq \frac{n}2$ and $t\leq \frac{n}2$.
In particular, if $s+t=n$, then $n$ is even and $s=t=\frac{n}2$.
\end{corollary}

The same authors also show that linear sets with complementary weights with only two points of weight greater than one admit an algebraic description: 

\begin{theorem} \label{th:relationfandg}\cite[Theorem 4.6]{napolitano2022linear}
Let $L_W$ be an $\fq$-linear set of $\PG(1,q^{2t})$ of rank $n=2t$ admitting two points $P$ and $Q$ of weight $t$. Then $L_W$ is $\mathrm{PGL}(2, q^n)$-equivalent to an $\fq$-linear set $L_U$ where
\[ U=T_{g,\eta} \times S_{f,\xi}=\{ (v+\eta g(v),u+\xi f(u)) \colon u,v \in \F_{q^t} \}, \]
with $\eta,\xi \in \F_{q^{2t}}\setminus\F_{q^t}$, $\xi^2=a\xi+b$, $\eta=A\xi+B$ with $a,b,A,B \in \F_{q^t}$ and $f(x),g(x) \in \mathcal{L}_{t,q}$.
Also, $P$ and $Q$ are the only points of $L_W$ with weight greater than one if and only if 
\begin{equation} \label{eq:relationfandg}
f(\alpha_0 v)+f(\alpha_1 A b g(v))+f(\alpha_0 B g(v))=\alpha_1 v+\alpha_0 A g(v)+\alpha_1 A a g(v)+\alpha_1 B g(v),
\end{equation}
has at most $q$ solutions in $v$ for every $\alpha_0,\alpha_1 \in \F_{q^t}$ with $(\alpha_0,\alpha_1)\ne(0,0)$.
\end{theorem}

Using this algebraic description, the authors of \cite{napolitano2022linear} were able to produce linear sets with complementary weights having exactly two points of weight greater than one.% from scattered polynomials.

\section{The weight distribution of the points in a linear set with complementary weights}

We will extend Theorem \ref{thm:numberopointsr} by giving a  description of the set of points with {\em any} given weight in a linear set with complementary weights.

\begin{definition}
Consider $T$ to be an $\fq$-subspace of $\fqn$. Let $\mathcal{L}_i(T)$ be the set of all the $i$-dimensional $\fq$-subspaces contained in $T$.
    \end{definition}

\begin{theorem}\label{thm:weightSxT}
Let $U=S\times T$, where $S$ is a $(k-r)$-dimensional $\fq$-subspace of $\fqn$ with $r \leq k-r$ and $T$ is an $r$-dimensional $\F_q$-subspace of $\F_{q^n}$.
The set of points of weight at least $i$ in $L_U$, different from $\langle (1,0)  \rangle_{\fqn}$, is 
\[ \left\{ \langle (\xi,1)\rangle_{\fqn} \colon \xi \in  \mathcal{I}_i(S,T)\right\}, \] where 

$$\mathcal{I}_i(S,T)=\bigcup_{\substack{a_1,\ldots,a_i \in T\\ \langle a_1,\ldots,a_i \rangle_{\fq} \in \mathcal{L}_i(T)}} a_1^{-1}S\cap \ldots \cap a_i^{-1} S.$$
\end{theorem}
\begin{proof}

Consider the map \begin{align*}
\Phi: \mathcal{I}_i(S,T)&\to \mathrm{PG}(1,q^n)\\
\xi&\mapsto P_\xi:=\langle (\xi,1)\rangle_{\F_{q^n}}.
\end{align*}

Let $\mathcal{P}_{\geq i}$ be the set of points in $L_U$, different from $\langle(1,0)\rangle_{\F_{q^n}}$ of weight at least $i$. We claim that $\Phi(\mathcal{I}_i(S,T))=\mathcal{P}_{\geq i}.$

Consider a point  $P_\xi \in \Phi(\mathcal{I}_i(S,T))$. Since $\xi\in \mathcal{I}_i(S,T),$ we have that $\xi \in a_1^{-1}S\cap \ldots \cap a_i^{-1} S$, for some $a_1,\ldots,a_i \in T$ that are $\fq$-linearly independent. This implies that there exist $b_1,\ldots,b_i \in S$ such that 
\[ \frac{b_1}{a_1}=\ldots=\frac{b_i}{a_i}=\xi. \]
The vectors $(b_1,a_1),\ldots,(b_i,a_i)$ belong to $U=S\times T$ and are $\F_q$-linearly independent since $a_1,\ldots,a_i$ are $\fq$-linearly independent. Since $\langle (b_j,a_j)\rangle_{\F_{q^n}}=\langle (\xi,1)\rangle_{\F_{q^n}}$ for all $j$, it follows that $U \cap \langle (\xi, 1) \rangle_{\fqn}$ is at least $i$-dimensional, so $P_\xi$ has weight at least $i$.
 It follows that $\Phi(\mathcal{I}_i(S,T))\subseteq \mathcal{P}_{\geq i}.$

Vice versa, consider a point $P\in \mathcal{P}_{\geq i}$, different from $\langle (1,0)\rangle_{\F_{q^n}}$, then $P$ can be written in a unique way as $P_\xi$ for some $\xi$ and  there exists $i$ vectors $(s_1,t_1),\ldots,(s_i,t_i) \in U\cap \langle (\xi, 1) \rangle_{\fqn}$ which are $\fq$-linearly independent.
 Moreover, the elements $t_1,\ldots,t_i$ are $\F_q$-linearly independent, since otherwise there would be a non-zero vector $(\sum_m\lambda_m s_m,0)\in U\cap \langle (\xi, 1) \rangle_{\fqn} $, a contradiction since the second coordinate of any non-zero vector in this intersection needs to be non-zero. 

We find that 
\[ \frac{s_1}{t_1}=\ldots=\frac{s_i}{t_i}=\xi, \]
which implies that $\xi \in t_1^{-1}S\cap\ldots\cap t_i^{-1}S$.

Since we have seen that $t_1,\ldots,t_i$ are $\fq$-linearly independent, $\xi\in \mathcal{I}_i(S,T),$ so $\mathcal{P}_{\geq i}\subseteq \mathcal{I}_i(S,T)$.

\end{proof}

Note that for $i=r$, the above result indeed coincides with Theorem \ref{thm:numberopointsr} as there is only one subspace of $T$ of dimension $r$.
As a corollary, we can now detect the points in $L_U$ with weight exactly $i$.

\begin{corollary}\label{cor:preciseThm3.2}
Let $U=S\times T$, where $S$ is a $(k-r)$-dimensional $\fq$-subspace of $\fqn$ with $r \leq k-r$ and $T$ is an $r$-dimensional $\F_q$-subspace of $\F_{q^n}$.
The set of points of weight exactly $i$ in $L_U$ different from $\langle (1,0)  \rangle_{\fqn}$ is 
\[ \left\{ \langle (\xi,1)\rangle_{\fqn} \colon \xi \in \mathcal{I}_i(S,T)\setminus \mathcal{I}_{i+1}(S,T) \right\}. \]
\end{corollary}

\begin{remark}
    The well-known connection between linear sets and rank-metric codes, see \cite{Randrianarisoa2020ageometric,alfarano2022linear,sheekeyVdV}, allows us to rephrase Theorem \ref{thm:weightSxT} and Corollary \ref{cor:preciseThm3.2} in terms of rank-metric codes. More precisely, suppose that $\mathcal{C}$ is a nondegenerate $\fqn$-linear rank-metric code in $\mathbb{F}_{q^n}^k$ of dimension two admitting a generator matrix of the form
    \[
    G=\begin{pmatrix}
        s_1 & \cdots & s_{k-r} & 0 & \cdots & 0 \\
        0 & \cdots & 0 & t_1 & \cdots & t_r
    \end{pmatrix}.
    \]
    Every codeword of $\mathcal{C}$ can be written as $(x_0,x_1)G$ for some $x_0,x_1 \in \fqn$.
    Clearly, the rank weight is the same for proportional codewords and
    \[ \mathrm{rk}_q((1,0)G)=k-r \,\,\text{and}\,\, \mathrm{rk}_q((0,1)G)=r. \]
    Therefore, up to proportionality, we can assume that a nonzero codeword $c \in \mathcal{C}$ not proportional to $(1,0)G$ nor to $(0,1)G$ is of the form $(1,-\xi)G$ for some $\xi \in \fqn$. By \cite[Theorem 2]{Randrianarisoa2020ageometric}, we have that
    \[ \mathrm{rk}_q((1,-\xi)G)=k-w_{L_{S\times T}}(\la (\xi,1)\ra_{\fqn}). \]
    By Corollary \ref{cor:preciseThm3.2}, the number { of codewords of rank weight $n-i$ is given by the size of the following set:}
    \[\left\{ \alpha (\xi,1) \colon \xi \in \mathcal{I}_i(S,T)\setminus \mathcal{I}_{i+1}(S,T), \alpha \in \fqn \right\}.\]
\end{remark}

Using that for any $a_1,a_2 \in \fqn^*$ 
$$\dim_{\fq}(a_1^{-1}S \cap a_2^{-1}S)=\dim_{\fq}(a_1a_2(a_1^{-1}S \cap a_2^{-1}S))=\dim_{\fq}(a_1S \cap a_2S),$$ 
we also get another characterisation for the linear sets of complementary weights with only two points of weight greater than or equal to two.

\begin{corollary}\label{weg}
Let $U=S\times T$, where $S$ is a $(k-r)$-dimensional $\fq$-subspace of $\fqn$ with $r \leq k-r$ and $T$ is an $r$-dimensional $\F_q$-subspace of $\F_{q^n}$.
Then $L_U$ has only two points of weight at least two if and only if 
\[ \dim_{\fq}(a_1S \cap a_2S)=0, \]
for any $a_1,a_2 \in T$ which are $\fq$-linearly independent.
\end{corollary}

We will now focus on the case $n=2t$, where there are two points of weight $t$. Let $\xi \in \mathbb{F}_{q^n}\setminus\mathbb{F}_{q^t}$ and $f(x) \in\mathcal{L}_{t,q}$ and let
\[ S_{f,\xi}=\{ u+\xi f(u) \colon u \in \mathbb{F}_{q^t} \}. \]

The following easy lemma shows that  the weight of the points of a linear set with complementary weights $L_{S\times T}$ can also be derived by finding the dimension of intersection of the subspaces $S$ and the multiplicative cosets of $T$ in $\fqn$.

\begin{lemma}\label{prop:weightpointSalphaT}
    Let $U=S\times T$, where $S$ is a $(k-r)$-dimensional $\fq$-subspace of $\fqn$ with $r \leq k-r$ and $T$ is an $r$-dimensional $\F_q$-subspace of $\F_{q^n}$.
    For any point $\la (1,\alpha)\ra_{\fqn}$, with $\alpha \in \fqn^*$,
    \[ w_{L_U}(\la (1,\alpha)\ra_{\fqn})=\dim_{\fq}(S\cap \alpha^{-1}T). \]
\end{lemma}
\begin{proof}
    By definition, we have that
    \[ w_{L_U}(\la (1,\alpha)\ra_{\fqn})=\dim_{\fq}(\{\rho \in \fqn \colon \rho(1,\alpha)\in S\times T\}), \]
    and noting that $\rho(1,\alpha)\in S\times T$ if and only if $\rho \in S\cap \alpha^{-1}T$, we obtain the assertion.
\end{proof}

As a consequence, when we have a polynomial description of $S$ and $T$, we have the following.

\begin{corollary}\label{cor:weightscomplpoli}
    Let $n=2t$, $\xi,\eta \in \mathbb{F}_{q^n}\setminus\mathbb{F}_{q^t}$ and $f(X),g(X) \in\mathcal{L}_{t,q}$.
    Let $\xi^2=A\xi+B$ and $\eta=a\xi+b$, for some $A,B,a,b \in \mathbb{F}_{q^t}$.
    The weight of $\la (1,\alpha) \ra_{\fqn}$, with $\alpha^{-1}=\alpha_0+\xi\alpha_1$, in $L_{S_{f,\xi}\times S_{g,\eta}}$ is 
 
 \[ \dim_{\fq}(\ker\left(f(\alpha_0 X+(\alpha_0 b+a\alpha_1B)g(X))-\alpha_1 X-(a\alpha_0+ a\alpha_1 A+b\alpha_1) g(X)\right)). \]  
    
\end{corollary}
\begin{proof}
    By the above proposition, we need to compute the dimension of $S_{f,\xi}\cap \alpha^{-1} S_{g,\eta}$.
    Then $x_0+\xi f(x_0) \in \alpha^{-1} S_{g,\eta}$ if and only if there exists $y_0 \in \F_{q^t}$ such that

     \[ x_0+\xi f(x_0)=(\alpha_0+\xi\alpha_1)(y_0+bg(y_0)+a\xi g(y_0)). \]
    Using that $\{1,\xi\}$ is an $\F_{q^t}$-basis of $\fqn$ we get that

     \[\left\{
    \begin{array}{ll}
    x_0=\alpha_0 y_0+\alpha_0 b g(y_0)+a g(y_0)\alpha_1B,\\
    f(x_0)= \alpha_0 a g(y_0)+\alpha_1 y_0+a g(y_0) \alpha_1 A+\alpha_1 b  g(y_0),
    \end{array}
    \right.\]
    from which we get that the weight of $\langle (1,\alpha)\rangle_{\F_{q^n}}$ is \[ \dim_{\fq}(\ker(f(\alpha_0 X)+f(\alpha_0 b g(X))+f(a g(X)\alpha_1B)-\alpha_0 a g(X)-\alpha_1 X-a g(X) \alpha_1 A-\alpha_1 b  g(X))). \] 
    Using that $f$ is additive, the statement now follows.
\end{proof}

\begin{remark}\label{eta=xi}
    If $\eta=\xi$ the polynomial in Corollary \ref{cor:weightscomplpoli} reduces to

     \[ f(\alpha_0 X+ \alpha_1Bg(X))-(\alpha_0+\alpha_1A) g(X)-  \alpha_1 X. \] 
\end{remark}

\section{Explicit linear sets of complementary type and their weight distributions}
In what follows, we will use the criteria derived in the previous section to present linear sets for which we can control the weight distribution.
\subsection{\texorpdfstring{ $S_f\times S_\Tr$}{f x Tr}}
\begin{proposition}\label{prop:tracewithf}
    Let $n=2t$, $\xi \in \mathbb{F}_{q^n}\setminus\mathbb{F}_{q^t}$ and $f(X) \in\mathcal{L}_{t,q}$. Then the point $\la (1,\alpha) \ra_{\fqn}$ of $L_{S_{f,\xi}\times S_{\mathrm{Tr}_{q^t/q},\xi}}$ has weight at most 
    \[\dim_{\fq}(\mathrm{Im}(f))+1\]
    if $\alpha\notin \F_{q^t}$ and exactly 
    % \[\dim_{\fq}(\ker(f(\alpha_0 X)-\alpha_0 \mathrm{Tr}_{q^t/q}(X))),\]
    \[\dim_{\fq}\left(\ker\left(f\left(\frac{1}{\alpha} X\right)-\frac{1}{\alpha} \mathrm{Tr}_{q^t/q}(X)\right)\right),\]
    if  $\alpha\in \F_{q^t}^*$.
\end{proposition}
\begin{proof}
    Let $\xi^2=A\xi+B$, for some $A,B \in \mathbb{F}_{q^t}$ and let $\alpha^{-1}=\alpha_0+\xi\alpha_1$.
    By Corollary \ref{cor:weightscomplpoli} and Remark \ref{eta=xi} we need to check the dimension of the kernel of the following polynomial
    \[G(X)=f(\alpha_0 X+\alpha_1B\mathrm{Tr}_{q^t/q}(X))-(\alpha_0+\alpha_1A) \mathrm{Tr}_{q^t/q}(X)-\alpha_1 X,\]
    for $\alpha_0,\alpha_1 \in \F_{q^t}$.
    Using that the image of $\Tr_{q^t/q}$ is equal to $\fq$, it follows that an element $\beta \in \ker(G)$ if and only if 
    \[ \alpha_1\beta=f(\alpha_0 \beta)+\mathrm{Tr}_{q^t/q}(\beta)(f(\alpha_1B)-\alpha_0 - \alpha_1 A), \]
    therefore 
    \[\alpha_1 \beta \in \mathrm{Im}(f(\alpha_0 X))+\epsilon \fq,\]
    where $\epsilon=f(\alpha_1B)-\alpha_0 - \alpha_1 A$.
    Hence, if $\alpha_1 \ne 0$ then the weight of $\la (1,\alpha) \ra_{\fqn}$, with $\alpha^{-1}=\alpha_0+\xi\alpha_1$, in $L_{S_{f,\xi}\times S_{\mathrm{Tr}_{q^t/q},\xi}}$ is at most $\dim_{\fq}(\mathrm{Im}(f))+1$. If $\alpha_1=0$ then 
    \[G(X)=f(\alpha_0 X)-\alpha_0 \mathrm{Tr}_{q^t/q}(X),\]
    and the assertion follows.
\end{proof}

\begin{lemma}\label{lem:tracewithinvertible}
    Let $n=2t$, $\xi \in \mathbb{F}_{q^n}\setminus\mathbb{F}_{q^t}$, $f(x) \in\mathcal{L}_{t,q}$, and $\alpha^{-1}=\alpha_0+\xi\alpha_1$. Suppose that $\psi_{\alpha_0,\alpha_1}(X):=f(\alpha_0 X)-\alpha_1X$ is invertible, then the point $\la (1,\alpha) \ra_{\fqn}$ of $L_{S_{f,\xi}\times S_{\mathrm{Tr}_{q^t/q},\xi}}$ has weight  
    at most one. 
\end{lemma}
\begin{proof}
    By Corollary \ref{cor:weightscomplpoli} and Remark \ref{eta=xi} we need to check the dimension of the kernel of $$G(X)=f(\alpha_0 X)-\alpha_1(X)+(f(\alpha_1B)-\alpha_0-\alpha_1A)\mathrm{Tr}_{q^t/q}(X),$$
 which we can rewrite as $G(X)=\psi_{\alpha_0,\alpha_1}(X)-\gamma \mathrm{Tr}_{q^t/q}(X)$ where we have assumed $\psi_{\alpha_0,\alpha_1}(X)$ to be invertible.
 Since $\mathrm{Im}(\gamma\mathrm{Tr}_{q^t/q}(X))$ is one-dimensional, and the elements of $\ker(G)$ are in $\psi_{\alpha_0,\alpha_1}^{-1}(\mathrm{Im}(\gamma\mathrm{Tr}_{q^t/q}(X)))$, we find that $\ker(G)$ is at most $1$-dimensional.
  \end{proof}

In the next result we consider the special case in which the polynomial $f$ is the trace function.

\begin{proposition}\label{prop:tracextrace}
    Let $n=2t$, $\xi \in \mathbb{F}_{q^n}\setminus\mathbb{F}_{q^t}$,  $q$ odd and $\xi^2=A\xi+B$ where $\mathrm{Tr}_{q^t/q}(A)\neq -2.$
    Then, for any point $\la (1,\alpha) \ra_{\fqn} \in L_{S_{\mathrm{Tr}_{q^t/q},\xi}\times S_{\mathrm{Tr}_{q^t/q},\xi}}$
    \[ w_{L_{S_{\mathrm{Tr}_{q^t/q},\xi}\times S_{\mathrm{Tr}_{q^t/q},\xi}}}(\la (1,\alpha) \ra_{\fqn})= \begin{cases}
        1, & \text{if } \alpha\notin \F_{q^t};\\
        t-2, & \text{if } \alpha\in \F_{q^t}\setminus \F_q;\\
        t, & \text{if } \alpha\in \F_q^*.
    \end{cases}\]
\end{proposition}
\begin{proof}
First suppose that $\alpha\notin \F_{q^t}$. Since $\dim_{\fq}(\mathrm{Im}(\mathrm{Tr}_{q^t/q}))=1$, it follows from Proposition \ref{prop:tracewithf} that the weight of a point $\langle(1,\alpha)\rangle$ with $\alpha\notin \F_{q^t}$ is at most $2$. Now assume to the contrary that there is a choice of $\alpha^{-1}=\alpha_0+\xi\alpha_1$ where $\alpha_1\neq 0$ such that $\langle(1,\alpha)\rangle$ has weight $2$. In view of Remark \ref{eta=xi}, this means that $\dim_{\F_q}(\ker(G))$ is exactly two, where \[G(X)=\mathrm{Tr}_{q^t/q}(\alpha_0X)+\Tr(\alpha_1B)\mathrm{Tr}_{q^t/q}(X)-(\alpha_0+\alpha_1A)\mathrm{Tr}_{q^t/q}(X)-\alpha_1X.\]

We first see that there are at most $q$ elements in $\ker(G)\cap T_0$, where $T_0$ is the set of elements $x$ with $\mathrm{Tr}_{q^t/q}(x)=0$: if $\mathrm{Tr}_{q^t/q}(x_0)=0$ and $G(x_0)=0$ then $\mathrm{Tr}_{q^t/q}(\alpha_0x_0)=\alpha_1x_0$, and therefore, $x_0=\frac{\lambda}{\alpha_1}$ for some $\lambda\in \F_q$. Therefore, if $\ker(G)$ is $2$-dimensional, there are at least $q^2-q=q(q-1)$ elements $x\in \ker(G)$ with $\mathrm{Tr}_{q^t/q}(x)\neq 0$, and therefore, since $\ker(G)$ is an $\F_q$-subspace, and $\mathrm{Tr}_{q^t/q}(\lambda x)=\lambda\mathrm{Tr}_{q^t/q}(x)$ for $\lambda\in \F_q$, there at least $q$ elements in $\ker(G)$ with $\mathrm{Tr}_{q^t/q}(x)=1$.
Now let $x_1$ and $x_2$ be two distinct elements with $\mathrm{Tr}_{q^t/q}(x_1)=\mathrm{Tr}_{q^t/q}(x_2)=1$ and $G(x_1)=G(x_2)=0$. It follows that $x_1-x_2\in \ker(G)\cap T_0$, and therefore, $x_1-x_2=\frac{\lambda}{\alpha_1}$ for some $\lambda\in \F_q$. It follows that $\mathrm{Tr}_{q^t/q}(\frac{1}{\alpha_1})=0.$
Since $\mathrm{Tr}_{q^t/q}(\alpha_0x_1)-\alpha_0-\alpha_1A+\mathrm{Tr}_{q^t/q}(\alpha_1B)=\alpha_1x_1$, we also see that $x_1=\frac{\mu-\alpha_0}{\alpha_1}-A$ for some $\mu\in \F_q$. 
Since there are at least $q$ different $x_1$'s and all of them are of the form $\frac{\mu-\alpha_0}{\alpha_1}-A,$ we find that all elements $r_\mu=\frac{\mu-\alpha_0}{\alpha_1}-A$ need to be in $\ker(G).$

It follows that $G(r_0)=0$, so $G(\frac{-\alpha_0}{\alpha_1})=G(A)$ 
but also $G(r_1)=0$, so $G(\frac{1-\alpha_0}{\alpha_1})=G(A)$, which forces $G(\frac{1}{\alpha_1})=0.$
Therefore, \[\mathrm{Tr}_{q^t/q}\left(\frac{\alpha_0}{\alpha_1}\right)+\mathrm{Tr}_{q^t/q}(\alpha_1B)\mathrm{Tr}_{q^t/q}\left(\frac{1}{\alpha_1}\right)-(\alpha_0+\alpha_1A)\mathrm{Tr}_{q^t/q}\left(\frac{1}{\alpha_1}\right)-1=0.\]
Hence, since $\mathrm{Tr}_{q^t/q}(\frac{1}{\alpha_1})=0$, $\mathrm{Tr}_{q^t/q}(\frac{\alpha_0}{\alpha_1})=1$. Since $x_1=\frac{\mu-\alpha_0}{\alpha_1}-A$ has $\mathrm{Tr}_{q^t/q}(x_1)=1$, it follows that $\mathrm{Tr}_{q^t/q}(A)=-2$, a contradiction.

    Now, suppose that $\alpha \in \F_{q^t}^*$. 
    Again, using Proposition \ref{prop:tracewithf}, we find that the weight of $\la (1,\alpha) \ra_{\fqn}$ is the dimension of the kernel of $$G(X)=\mathrm{Tr}_{q^t/q}\left(\frac{1}{\alpha} X\right)-\frac{1}{\alpha} \mathrm{Tr}_{q^t/q}(X).$$ If $\alpha\in \F_q^*$, this polynomial is identically zero so $\dim(\ker(G))=t$. 
    We may suppose that $\alpha\in \F_{q^t}\setminus \F_q$. Consider  $\beta\in \ker(G)$, then we have that $\mathrm{Tr}_{q^t/q}(\frac{1}{\alpha}\beta)=\frac{1}{\alpha}\mathrm{Tr}_{q^t/q}(\beta)$. The left hand side of this equality belongs to $\F_q$, hence, since $\alpha\notin \F_q$, $\mathrm{Tr}_{q^t/q}(\beta)=0=\mathrm{Tr}_{q^t/q}(\frac{1}{\alpha}\beta)$. It follows that $\ker(G)\subseteq\ker(\mathrm{Tr}_{q^t/q}(X))\cap \ker(\mathrm{Tr}_{q^t/q}(\frac{1}{\alpha}X))$, and it is clear that $\ker(\mathrm{Tr}_{q^t/q}(X))\cap \ker(\mathrm{Tr}_{q^t/q}(\frac{1}{\alpha}X))\subseteq \ker(G)$. Since $\ker(\frac{1}{\alpha}\mathrm{Tr}_{q^t/q}(X))$ and $\ker(\mathrm{Tr}_{q^t/q}(\frac{1}{\alpha}X))$ are two distinct $(t-1)$-dimensional $\F_q$-subspaces of $\F_{q^t}$, their intersection is $(t-2)$-dimensional.
\end{proof}

In the next result we see how the weight distribution dramatically changes when considering $f(X)=X^q$.

\begin{proposition}\label{lem:tracewithxq}
    Let $n=2t$, $\xi \in \mathbb{F}_{q^n}\setminus\mathbb{F}_{q^t}$. Then, for any $\alpha \in \fqn$ 
    \[ 
    w_{L_{S_{x^q,\xi}\times S_{\mathrm{Tr}_{q^t/q},\xi}}} (\la (1,\alpha) \ra_{\fqn})\leq 
    \begin{cases}
        1, & \text{if } \alpha\in \F_{q^t}^*;\\
        2, & \text{if } \alpha\in \F_{q^n}^*.
    \end{cases}
    \] 
\end{proposition}
\begin{proof}
First suppose that $\alpha\in \F_{q^t}^*.$ We know from Proposition \ref{prop:tracewithf} that we need to determine $\dim_{\fq} (\ker(H))$ where \[H(X)=\left(\frac{1}{\alpha} X\right)^q-\frac{1}{\alpha} \mathrm{Tr}_{q^t/q}(X).\]
If $x\in \ker(H)$, then $\alpha^{1-q}x^q=\mathrm{Tr}_{q^t/q}(x)\in \fq.$ It immediately follows that $\dim(\ker(H_\alpha))\leq 1.$
  Now suppose that $\alpha\notin\F_{q^t}.$  We know from Corollary \ref{cor:weightscomplpoli} that we need to determine the dimension of $\ker(G(X))$ where
    $$G(X)=\alpha_0^qX^q-\alpha_1 X+(\alpha_1^qB^q-\alpha_0-\alpha_1^qA^q)\Tr_{q^t/q}(X),$$ and $\alpha^{-1}=\alpha_0+\xi\alpha_1.$

We will show that for all choices of $(\alpha_0,\alpha_1)$, different from $(0,0)$, $\ker(G)$ is at most $2$-dimensional.  Write $G(X)=\psi_{\alpha_0,\alpha_1}(X)-\gamma \mathrm{Tr}_{q^t/q}(X)$, where $\psi_{\alpha_0,\alpha_1}(X)=\alpha_0^qX^q-\alpha_1X$ and $\gamma=\alpha_1^qB^q-\alpha_0-\alpha_1^qA^q.$

 We know from Lemma \ref{lem:tracewithinvertible} that if $\psi_{\alpha_0,\alpha_1}(X)$ is invertible, then   $\ker(G)$ is at most $1$-dimensional so the associated point $\langle(1,\alpha)\rangle_{\F_{q^n}}$, has weight at most $1$. 
 Note that if $\alpha_0=0$, $\psi_{\alpha_0,\alpha_1}(X)$ is invertible, so if $\psi_{\alpha_0,\alpha_1}$ is not invertible, it is a $q$-polynomial with $q$-degree $1$, and therefore it has exactly $q$ roots and $\dim_{\fq}(\ker(\psi_{\alpha_0,\alpha_1}))=1$.% and $\dim(\mathrm{im}(\psi_{\alpha_0,\alpha_1}))=t-1$.
 
Since $\psi_{\alpha_0,\alpha_1}(x)=\gamma \Tr_{q^t/q}(x)\in \gamma \F_q$ for all $x\in \ker(G)$, and $\psi_{\alpha_0,\alpha_1}$ is $q$-to-$1$, there are at most $q^2$ elements in $\ker(G),$ so $\dim_{\fq}(\ker(G))\leq 2.$
 \end{proof}

\subsection{\texorpdfstring{$S_{x^{q^s}}\times S_{x^{q^s}}$}{frob x frob}}

In this section, we study the structure and weight distribution of certain linear sets of the form 
$L_{S_{f,\xi} \times S_{f,\xi}}$ in $\PG(1,q^n)$, where $f$ is either monomial or binomial. 
In the following propositions, we use the fact that if a polynomial defined on $\F_{q^{t}}$ of the form $a_0x+a_1x^{q^s}+a_2 x^{q^{2s}}+\ldots+a_{k-1}x^{q^{s(k-1)}}-x^{q^{sk}}$ splits completely then 
\begin{equation}\label{eq:normcond}
\mathrm{N}_{q^t/q}(a_0)=(-1)^{t(k+1)}
\end{equation} 
(see e.g. \cite[Corollary 3.1]{csajbokpolynomial}). 

\begin{proposition}\label{prop:powerqxpowerq}
    Let $n=2t$, $\xi \in \mathbb{F}_{q^n}\setminus\mathbb{F}_{q^t}$ and $s \in \mathbb{N}$ with $\gcd(s,t)=1$. Then the point $\la (1,\alpha) \ra_{\fqn}$ of $L_{S_{x^{q^s},\xi}\times S_{x^{q^s},\xi}}$ has weight at most two if $\alpha\notin \F_{q^t}$; has weight $0$ if $\alpha \in \F_{q^t}\setminus\fq$ and has weight $t$ if $\alpha \in \fq^*$.

   Furthermore, if we let $\xi^2=A\xi+B$, for some $A,B \in \mathbb{F}_{q^t}$, and 
    if $\mathrm{N}_{q^t/q}(B)\neq(-1)^t$, then all points $\la (1,\alpha) \ra_{\fqn}$, $\alpha\notin \F_{q^t}$ of $L_{S_{x^{q^s},\xi}\times S_{x^{q^s},\xi}}$ have weight one. 
\end{proposition}
\begin{proof}
    Let $\xi^2=A\xi+B$, for some $A,B \in \mathbb{F}_{q^t}$.
    We know from Corollary \ref{cor:weightscomplpoli} and Remark \ref{eta=xi} that the weight of $\la (1,\alpha) \ra_{\fqn}$, with $\alpha^{-1}=\alpha_0+\xi\alpha_1$, in $L_{S_{x^{q^s},\xi}\times S_{x^{q^s},\xi}}$ is the dimension of the kernel of
    \[ G(X)=\alpha_0^{q^s} X^{q^s}+X^{q^{2s}}\alpha_1^{q^s}B^{q^s}-(\alpha_0+\alpha_1 A) X^{q^s}-\alpha_1 X. \]
    If $\alpha_1\ne 0$ then $w_{L_{S_{x^{q^s},\xi}\times S_{x^{q^s},\xi}}}(\la (1,\alpha) \ra_{\fqn})\leq 2$, since $G$ is a nonzero linearised polynomial of $q^s$-degree two.
    Moreover, if this polynomial splits completely then $\mathrm{N}_{q^t/q}(\frac{\alpha_1}{\alpha_1^{q^s}B})= (-1)^t$, which happens if only if $\mathrm{N}_{q^t/q}(B)=(-1)^t$ (a condition which is independent from $\alpha_1$).
    
    If $\alpha_1=0$ and $\alpha_0 \notin \fq$ then $G(X)=\alpha_0^{q^s} X^{q^s}-\alpha_0 X^{q^s}$, which is invertible; therefore $\dim(\ker(G))=0$.
    
    Finally, if $\alpha_1=0,\alpha_0\in \F_q$, then $G$ is identically zero, and therefore $\dim_{\fq}(\ker(G))=t$.
\end{proof}

In the following, we consider the binomial of Lunardon-Polverino type, \cite{lunardon2001blocking}.

\begin{proposition}\label{prop:LPxLP}
    Let $n=2t$, $t\geq 5$, $\xi \in \mathbb{F}_{q^n}\setminus\mathbb{F}_{q^t}$, $\delta \in \F_{q^t}$ with $\N_{q^t/q}(\delta)^2\ne 1$ and $s \in \mathbb{N}$ with $\gcd(s,t)=1$. Then, 
     \[ 
    w_{L_{S_{x^{q^s}+\delta x^{q^{s(t-1)}},\xi}\times S_{x^{q^s}+\delta x^{q^{s(t-1)}},\xi}}} (\la (1,\alpha) \ra_{\fqn})\leq 
    \begin{cases}
        3, & \text{if } \alpha\notin \F_{q^t};\\
        2, & \alpha\in \F_q^t\setminus \F_q,
    \end{cases}
    \] 
    and it has weight exactly $t$ if and only if $\alpha\in \F_q^*$.
\end{proposition}
\begin{proof}
    Let $\xi^2=A\xi+B$, for some $A,B \in \mathbb{F}_{q^t}$.
    Using Corollary \ref{cor:weightscomplpoli} and Remark \ref{eta=xi}, we know that the weight of $\la (1,\alpha) \ra_{\fqn}$, with $\alpha^{-1}=\alpha_0+\xi\alpha_1$, in $L_{S_{f,\xi}\times S_{g,\eta}}$ is the dimension of the kernel of
    \[ G(X)=\alpha_0^{q^s} X^{q^s}+\delta \alpha_0^{q^{s(t-1)}} X^{q^{s(t-1)}}+(X^{q^{2s}}+\delta^{q^{s}} X) \alpha_1^{q^s}B^{q^s}+\delta \alpha_1^{q^{s(t-1)}}B^{q^{s(t-1)}} (X+\delta^{q^{s(t-1)}} X^{q^{s(t-2)}})\] \[-(\alpha_0+\alpha_1A) (X^{q^s}+\delta X^{q^{s(t-1)}})-\alpha_1 X. \]
    If we now consider $G(X)^{q^{2s}}$, we obtain a linearised polynomial of $q^s$-degree at most $4$:
    \[ G(X)^{q^{2s}}=\alpha_0^{q^{3s}} X^{q^{3s}}+\delta^{q^{2s}} \alpha_0^{q^{s}} X^{q^{s}}+(X^{q^{4s}}+\delta^{q^{3s}} X^{q^{2s}}) \alpha_1^{q^{3s}}B^{q^{3s}}+\delta^{q^{2s}} \alpha_1^{q^{s}}B^{q^{s}} (X^{q^{2s}}+\delta^{q^{s}} X)\] \[-(\alpha_0^{q^{2s}}+\alpha_1^{q^s} A^{q^s}) (X^{q^{3s}}+\delta^{q^{2s}} X^{q^{s}})-\alpha_1^{q^{2s}} X^{q^{2s}}. \]
    Note that the coefficient of $X^{q^{4s}}$ is $\alpha_1^{q^{3s}}B^{q^{3s}}$ and the coefficient of $X$ is $\alpha_1^{q^s}B^{q^s}\delta^{q^{2s}+q^s}$.
    If $\alpha_1\ne 0$ then
    \[ \N_{q^t/q}(\alpha_1^{q^{3s}}B^{q^{3s}}) \ne \N_{q^t/q}(\alpha_1^{q^s}B^{q^s}\delta^{q^{2s}+q^s}), \]
    and so, by \eqref{eq:normcond},
    \[\dim_{\fq}(\ker(G))=\dim_{\fq}(\ker(G^{q^{2s}}))\leq 3.\]
    If $\alpha_1=0$ and $\alpha_0 \notin \fq$ then $G(X)^{q^{2s}}=\alpha_0^{q^{3s}} X^{q^{3s}}+\delta^{q^{2s}} \alpha_0^{q^{s}} X^{q^{s}} -\alpha_0^{q^{2s}} (X^{q^{3s}}+\delta^{q^{2s}} X^{q^{s}})=(\alpha_0^{q^{3s}}-\alpha_0^{q^{2s}})X^{q^{3s}}+\delta^{q^{2s}}(\alpha_0^{q^{s}}-\alpha_0^{q^{2s}})X^{q^{s}}$ is a nonzero polynomial whose kernel may have dimension at most two (as we can see it as a polynomial in the variable $X^{q^s}$ with $q^s$-degree less than or equal to two), and is the zero polynomial if $\alpha_0\in \F_q$.
\end{proof}

\subsection{\texorpdfstring{$S_f\times S_f$}{f x f}}

We have seen that for some choices of $S,T$, we can control the weight distribution of $L_{S\times T}$. This seems much more difficult when we consider the product of arbitrary subspaces. In the case of linear sets defined by the Cartesian product of a subspace with itself, we can derive the following (weak) statement. 
\begin{proposition}\label{prop:fxf}
    Let $n=2t$, $\xi\in \mathbb{F}_{q^n}\setminus\mathbb{F}_{q^t}$. Then the point $\la (1,\alpha) \ra_{\fqn}$ of $L_{S_{f,\xi}\times S_{f,\xi}}$ has weight at most $3\dim_{\fq}(\mathrm{Im}(f))$ if $\alpha\notin \F_{q^t}$, exactly  $\dim_{\fq}(\ker(f(\frac{1}{\alpha}X)-\frac{1}{\alpha}f(X))$ if $\alpha\in \F_q^t\setminus \F_q$ and exactly $t$ if $\alpha\in \F_q^*$.
\end{proposition}
\begin{proof}
    By Corollary \ref{cor:weightscomplpoli} and Remark \ref{eta=xi}, we know that we need to find the dimension of the kernel of
\[G(X)=f(\alpha_0X)+f(\alpha_1Bf(X))-(\alpha_0+\alpha_1A)f(X)-\alpha_1X,\]
where $\alpha^{-1}=\alpha_0+\xi\alpha_1$. If $\alpha_1=0$, that is, if $\alpha\in \F_{q^t}^*$, then $G(X)=f(\frac{1}{\alpha}X)-\frac{1}{\alpha}f(X),$ which shows the second part of the statement. We also see that  if $\alpha\in \F_q$, $G(X)$ is identically zero, so the associated point has weight $t$.
If $\alpha_1\neq 0$, then we see that for $x\in \ker(G),$
$x=\frac{1}{\alpha_1}(f(\alpha_0x)+f(\alpha_1Bf(x))-(\alpha_0+\alpha_1A)f(x)),$ and hence, if $\dim_{\fq}(\mathrm{Im}(f))=k$ then $\dim_{\fq}(\ker(G))\leq 3k.$
\end{proof}

\begin{remark}
Although Proposition~\ref{prop:fxf} provides only a rather weak upper bound on the weights, it can still yield meaningful information when applied to polynomials $f$ having a large kernel (and hence a small image). In such cases, the bound becomes significantly sharper and allows us to identify families of linear sets with constrained weight distributions.
\end{remark}

\section{\texorpdfstring{Linear sets determined by $\F_{q^t}\times f$}{Linear sets determined by Fqt x f}}

\subsection{\texorpdfstring{The weight distribution of $L_{\F_{q^t}\times S_{f,\xi}}.$}{The weight distribution of linear sets determined by Fqt x f}}

In the next theorem, we show how the weight distribution of a linear set $L_{\mathbb{F}_{q^t}\times S_{f,\xi}}$ in $\PG(1,q^{2t})$ is entirely determined by the weight distribution of the linear set $L_f$ in $\PG(1,q^t)$: we find two points of weight $t$, and the weights of the other points are in one-to-one  correspondence with the weights of the points in $L_f=\{\langle (x,f(x)\rangle_{\F_{q^t}}\mid x\in \F_{q^t}^*\}.$ 

We will use this to provide an iterative construction of linear sets with many different weights.

\begin{theorem}\label{thm:bigfromsmall}
Let $n=2t$, $\xi \in \mathbb{F}_{q^n}\setminus\mathbb{F}_{q^t}$ and $f(x) \in\mathcal{L}_{t,q}$. Suppose that $\langle (1,0)\rangle_{\mathbb{F}_{q^t}}\notin L_f$. Then the $\fq$-linear set $L_{\mathbb{F}_{q^t}\times S_{f,\xi}}$ has the following properties:
\begin{enumerate}
    \item $\mathrm{rk}(L_{\mathbb{F}_{q^t}\times S_{f,\xi}})=n$;
    \item $w_{L_{\mathbb{F}_{q^t}\times S_{f,\xi}}}(\langle (1,0)\rangle_{\mathbb{F}_{q^n}})=w_{L_{\mathbb{F}_{q^t}\times S_{f,\xi}}}(\langle (0,1)\rangle_{\mathbb{F}_{q^n}})=t$;
    \item $w_{L_{\mathbb{F}_{q^t}\times S_{f,\xi}}}(\langle (1,m)\rangle_{\mathbb{F}_{q^n}})=w_{L_f}(\langle (m_0,m_1)\rangle_{\mathbb{F}_{q^t}})$, where $m=m_0+m_1\xi$, for some $m_0,m_1 \in \mathbb{F}_{q^t}$;
    \item Let $W_{L_f}=\sum_{w=1}^t A_w X^w$ be the weight enumerator of the linear set $L_f$, where $A_w$ is the number of points of weight $w$. Then $W_{L_{\mathbb{F}_{q^t}\times S_{f,\xi}}}=2 X^t+(q^t-1)W_{L_f}.$
\end{enumerate}
\end{theorem}
\begin{proof}
    The first two points are easy to check.
    Now, suppose that $w_{L_{\mathbb{F}_{q^t}\times S_{f,\xi}}}(\langle (1,m)\rangle_{\mathbb{F}_{q^n}})=j\geq 1$ then there exist $\alpha_1,\ldots,\alpha_j,x_1,\ldots,x_j \in \mathbb{F}_{q^t}$ such that
    \[ \langle (\alpha_1,x_1+\xi f(x_1)),\ldots,(\alpha_j,x_j+\xi f(x_j))\rangle_{\mathbb{F}_q} = (\mathbb{F}_{q^t}\times S_{f,\xi}) \cap \langle (1,m)\rangle_{\mathbb{F}_{q^n}}. \]
    It immediately follows that for any $i \in \{1,\ldots,j\}$ we have
    \[ \alpha_i^{-1}(x_i+\xi f(x_i))=m, \]
    and since the $\alpha_i$'s are in $\mathbb{F}_{q^t}$ and $\{1,\xi\}$ is an $\mathbb{F}_{q^t}$-basis of $\mathbb{F}_{q^n}$ we have that
    \[ m_0=\alpha_i^{-1} x_i\,\,\,\text{and}\,\,\,m_1=\alpha_i^{-1} f(x_i). \]
    Therefore,
    \[\langle (m_0,m_1)\rangle_{\mathbb{F}_{q^t}}=\langle (\alpha_i^{-1}x_i,\alpha_i^{-1}f(x_i))\rangle_{\mathbb{F}_{q^t}}=\langle (x_i,f(x_i))\rangle_{\mathbb{F}_{q^t}},\]
    for any $i \in \{1,\ldots,j\}$.
    Note also that $x_1,\ldots,x_j$ are $\fq$-linearly independent. Indeed, if there exist $a_1,\ldots,a_j \in \mathbb{F}_{q}$ such that 
    \[ a_1x_1+\ldots+a_jx_j=0, \]
    then 
    \[ a_1(\alpha_1,x_1+\xi f(x_1))+\ldots+a_j (\alpha_j,x_j+\xi f(x_j))=(a_1\alpha_1+\ldots+a_j\alpha_j,0) \in (\mathbb{F}_{q^t}\times S_{f,\xi}) \cap \langle (1,m)\rangle_{\mathbb{F}_{q^n}},\]
    implying that $a_1\alpha_1+\ldots+a_j\alpha_j=0.$
     It follows that \begin{align*}\sum\nolimits_{i=1}^ja_i\left(\alpha_i,x_i+\xi f(x_i)\right)=\left(\sum\nolimits_{i=1}^j a_i \alpha_i,\sum\nolimits_{i=1}^j a_ix_i +\sum\nolimits_{i=1}^j a_if(x_i)\right)=\\\left(\sum\nolimits_{i=1}^j a_i \alpha_i,\sum\nolimits_{i=1}^j a_ix_i + f\left(\sum\nolimits_{i=1}^j a_ix_i\right)\right)=(0,0),\end{align*} and therefore, the $j$ vectors $(\alpha_i,x_i+\xi f(x_i))$, $i=1,\ldots, j$ are not spanning a $j$-dimensional space.
    
    Therefore, we have that 
    \[w_{L_{\mathbb{F}_{q^t}\times S_{f,\xi}}}(\langle (1,m)\rangle_{\mathbb{F}_{q^n}})\leq w_{L_f}(\langle (m_0,m_1)\rangle_{\mathbb{F}_{q^t}}).\]
    Now, suppose that $w_{L_f}(\langle (m_0,m_1)\rangle_{\mathbb{F}_{q^t}})=\ell$. As before this means that there exist $x_1,\ldots,x_{\ell} \in \mathbb{F}_{q^t}$ such that
    \[ \langle (x_1,f(x_1)),\ldots, (x_\ell,f(x_\ell))\rangle_{\fq}=U_f \cap \langle (m_0,m_1)\rangle_{\mathbb{F}_{q^t}}, \]
 where $U_f=\{\langle(x,f(x))\rangle_{\F_{q^t}}\mid x\in \F_{q^t}\}$.
 
 Note that $m_0$ and $m_1$ cannot be zero as $\langle (0,1)\rangle_{\mathbb{F}_{q^t}}, \langle (1,0)\rangle_{\mathbb{F}_{q^t}} \notin L_f$. So, we have that $m_1/m_0=f(x_i)/x_i$ for any $i \in \{1,\ldots,\ell\}$.
    Define $\beta_i=x_i/m_0=f(x_i)/m_1$. 
    We see that $(\beta_i,x_i+\xi f(x_i))=(\beta_i,m_0\beta_i+\xi m_1\beta_i)=\beta_i(1,m).$
    Therefore,
    \[ (\beta_1,x_1+\xi f(x_1)),\ldots,(\beta_\ell,x_\ell+\xi f(x_\ell)) \in (\mathbb{F}_{q^t}\times S_{f,\xi}) \cap \langle (1,m)\rangle_{\mathbb{F}_{q^n}},\]
    and arguing as before one can see that the $(\beta_i,x_i+\xi f(x_i))$'s are $\fq$-linearly independent. So, 
    \[w_{L_{\mathbb{F}_{q^t}\times S_{f,\xi}}}(\langle (1,m)\rangle_{\mathbb{F}_{q^n}})\geq w_{L_f}(\langle (m_0,m_1)\rangle_{\mathbb{F}_{q^t}}),\]
    and the statement of the third bullet point follows.

  Finally, we claim that the number of points in $L_{\mathbb{F}_{q^t}\times S_{f,\xi}}$, different from $\langle(0,1)\rangle_{\F_{q^{n}}}$ and $\langle(1,0)\rangle_{\F_{q^n}}$ of weight $w$ is $q^t-1$ times the number of points of weight $w$ in $L_f$. Suppose that $\langle (1,m)\rangle_{\F_{q^n}}$ and $\langle (1,m')\rangle_{\F_{q^n}}$ correspond to the same point $\langle (m_0,m_1)\rangle_{\F_{q^t}}$. Then $m=m_0+m_1\xi$ and $m'=m_0'+m_1'\xi$ where $m_0=\lambda m_0'$ and $m_1=\lambda m_1'$ for some $\lambda\in \F_{q^t}$ and it follows that $m'=\lambda m$ where $\lambda\in \F_{q^t}^*$. Hence, a point $\langle (m_0,m_1)\rangle_{\F_{q^t}}$ gives rise to $q^t-1$ different points $\langle (1,m)\rangle_{\F_{q^n}}$. Vice versa, if $m'=\lambda m$ where $\lambda\in \F_{q^t}^*,$ then $\langle (m_0,m_1)\rangle_{\F_{q^t}}=\langle (m'_0,m'_1)\rangle_{\F_{q^t}}$. The statement follows.
\end{proof}

\begin{remark}
    The above result extends  \cite[Corollary~4.7]{napolitano2022linear}, which was stated for scattered polynomials (see also \cite{longobardi2024scattered}).
\end{remark}

\subsection{\texorpdfstring{A linear set with many different weights in $\PG(1,q^{2^m})$ and its associated set of even type}{A linear set with many different weights and its associated set of even type}}

Using Theorem \ref{thm:bigfromsmall} we can iteratively construct linear sets in $\PG(1,q^{2t})$ from a linear set in $\PG(1,q^t),$ while keeping control over the weight distribution of the obtained linear set.

\begin{example}
Starting from a scattered linear set $L_f$ of rank $2$ in $\PG(1,q^2)$ not containing $\langle(0,1)\rangle_{\F_{q^2}}$ (that is, an $\F_q$-subline), Theorem \ref{thm:bigfromsmall} shows that $L_{\F_{q^t}\times S_{f,\xi}}$ is a linear set of rank $4$ in $\PG(1,q^4)$ with complementary weights which has two points of weight $2$ and all other points weight $1$. The total number of points in this linear set is $q^3+q^2-q+1=2+(q^2-1)(q+1)$. We know that $W_{L_f}=(q+1)X$ so indeed, $W_{L_{\F_{q^t}\times S_{f,\xi}}}=2X^2+(q^2-1)W_{L_f}=2X^2+(q^2-1)(q+1)X.$
\end{example}

If we iterate the construction of Theorem \ref{thm:bigfromsmall}, we find the following corollary, where we use $\Psi(L_f)$ for the  linear set $L_{\F_{q^t}\times S_{f,\xi}}$ in $\PG(1,q^{2t})$ corresponding to   $L_f=\{\langle (x,f(x))\rangle_{\F_{q^t}}|x\in \F_{q^t}^*\}$ be a linear set in $\PG(1,q^t).$

\begin{corollary}\label{cor:iterating}

   If $W_2(X)$ is the weight distribution of a linear set $L_f$ in $\PG(1,q^2)$, then the weight distribution of the linear set $\Psi^{m-1}(L_f)$, which is contained in $\PG(1,q^{2^{m}})$ is given by
   \begin{align*}
   W_{2^{m}}(X)=2X^{2^{m-1}}+2(q^{2^{m-1}}-1)X^{2^{m-2}}+2(q^{2^{m-1}}-1)(q^{2^{m-2}}-1)X^{2^{m-3}}+\ldots\\
   +2(q^{2^{m-1}}-1)(q^{2^{m-2}}-1)\cdots(q^4-1)X^2+(q^{2^{m-1}}-1)(q^{2^{m-2}}-1)\cdots(q^4-1)(q^2-1)W_2(X).
   \end{align*}
   The total size of the linear set is
   \[|L_f|\cdot\prod_{i=1}^{m-1}(q^{2^i}-1)+2\sum_{k=1}^{m-1}\prod_{i=k+1}^{m-1}(q^{2^i}-1).\]
\end{corollary}

Using that $W_{L_f}(X)=(q+1)X$ when $L_f$ is a subline, we find the following for $q=2$:

\begin{corollary}\label{cor:manyweights}
    
Let $L_f$ be a subline in $\PG(1,2^2)$. Then $\Psi^{m-1}(L_f)$ is contained in $\PG(1,2^{2^{m}})$, has $3\prod_{i=1}^{m-1}(2^{2^i}-1)$ points of weight $1$ and has 
$2\prod_{i=k+1}^{m-1}(2^{2^i}-1)$ points of weight $2^k$, where $1\leq k\leq m-1$ and the empty product equals $1$. The size of $\psi_{m-1}(L_f)$ is $\prod_{i=0}^{m-1}(2^{2^i}-1)+2\sum_{k=0}^{m-1}\prod_{k+1}^{m-1}(2^{2^i}-1).$
\end{corollary}
\begin{proof}
    We use that the total number of points is
    \begin{align*}
    (q+1)\prod_{i=1}^{m-1}(q^{2^i}-1)+2\sum_{k=1}^{m-1}\prod_{k+1}^{m-1}(q^{2^i}-1)=&\\
    (q-1)\prod_{i=1}^{m-1}(q^{2^i}-1)+2\Pi_{i=1}^{m-1}(q^{2^i}-1)+2\sum_{k=1}^{m-1}\prod_{k+1}^{m-1}(q^{2^i}-1)=&\\
    \prod_{i=0}^{m-1}(q^{2^i}-1)+2\sum_{k=0}^{m-1}\prod_{k+1}^{m-1}(q^{2^i}-1).
    \end{align*}
\end{proof}

\subsection{On certain sets of even type}

A {\em set of even type} in $\PG(2,q)$ is a set of points such that any line intersects the set in an even number of points. It is easy to see that the number of points in an even set is even, and that even sets can only exist in projective planes of even order.

It is well-known that using certain $\F_2$-linear sets in $\PG(1,q^t)$, $q$ even, one can construct translation sets which form {\em translation hyperovals}, or {\em translation KM-arcs} \cite{de2016linear}. Using the same construction, we can create small sets of even sets of type $(0,\ 2,\ \sqrt{q},\ 2\sqrt{q}-2)$ in $\PG(2,q)$, $q$ an even square (see Corollary \ref{cor:evensetfewweights}), as well as even sets of type $(0,2,4,8,\ldots, 2^{m-1})$ in $\PG(2,2^{2^m})$ (see Corollary \ref{cor:evenset}),

\begin{lemma}\label{lem:translationset}
    Let $L_g=\langle (x,g(x))\rangle_{\fq}\mid x\in \F_{q}\}$ be an $\F_2$-linear set in $\PG(1,q)$, $q=2^s$, which has points of weight $w_1,\ldots,w_k$  let $\mathcal{S}=\{\langle(1,x,g(x))\rangle_{\fq}\mid x\in \F_q^* \}$ and let $\mathcal{S}_\infty$ be the complement of the set $\{\langle(0,x,g(x))\rangle_{\fq}\mid x\in \F_q^* \}$. Then $\overline{\mathcal{S}}=\mathcal{S}\cup \mathcal{S}_{\infty}$ is a set of even type in $\PG(2,q)$ of size $2q+1-|L_g|$ that has line intersection sizes $0,2^{w_1},\ldots,2^{w_k}$ and $q+1-|L_g|$.
\end{lemma}
\begin{proof}
    Since $g(x)$ is an $\F_q$-linear map, $\mathcal{S}_\infty$ is precisely the set of points that are not a direction determined by the set $\mathcal{S}$. Recall that the size of the set $\mathcal{S}$ is $q$ so it follows that every line with a direction inside $\mathcal{S}_\infty$ meets $\mathcal{S}$ in exactly one point. Hence, every line with direction inside the set $\mathcal{S}_\infty$ is a $2$-secant to the set $\overline{\mathcal{S}}$. Furthermore, a line with direction $\langle (x,g(x)\rangle_{\fq}$, which has weight $w_j$ in the linear set $L_g$ meets the set $\mathcal{S}$ in $0$ or $2^{w_j}$ points.
    Finally, the line $X=0$ intersects the set $\overline{\mathcal{S}}$ in a set of size $|\mathcal{S}_\infty|=q+1-|L_g|$ and the size of $\overline{\mathcal{S}}$ equals $q+(q+1-|L_g|)$
\end{proof}

\begin{corollary}\label{cor:evensetfewweights}
    There exists a set of even type in $\PG(2,q)$ of size $q+2\sqrt{q}-2$ that has line intersection sizes $0,2,\sqrt{q}$ and $2(\sqrt{q}-1)$.
\end{corollary}
\begin{proof}
    Choosing $f$ as a scattered polynomial, the linear set $L=L_{\F_{\sqrt{q}}\times S_f,\xi}$ has exactly two points of weight $\sqrt{q}$ and all others of weight $1$  by Theorem \ref{thm:bigfromsmall}. It follows that $L$ has size $q-2\sqrt{q}+3$, so using $L$ in Lemma \ref{lem:translationset} shows the statement.
\end{proof}

\begin{corollary}\label{cor:evenset}
Let $q=2^{2^m}$.
    There exists a set of even type in $\PG(2,q)$ of size $2q+1-\prod_{i=0}^{m-1}(2^{2^i}-1)-2\sum_{k=0}^{m-1}\prod_{k+1}^{m-1}(2^{2^i}-1)$ that has line intersection sizes $0,2,2^2,2^{2^2},\ldots, 2^{2^{m-1}}$ and $q+1-\prod_{i=0}^{m-1}(2^{2^i}-1)-2\sum_{k=0}^{m-1}\prod_{k+1}^{m-1}(2^{2^i}-1).$
\end{corollary}
\begin{proof}
    Let $L_g=\langle (x,g(x))\rangle_{\F_{2^{2^{m}}}}\mid x\in \F_{2^{2^{m}}}\}$ be the point set of $\Psi^m(L_f)$ where $L_f$ is a subline in $\PG(1,2^2)$. Then $L_g$ has points of weight $0,1,2,2^2,\ldots,2^{m-1}$ by Corollary \ref{cor:manyweights}.
    Moreover, the size of $L_g$ is $2q+1-\prod_{i=0}^{m-1}(2^{2^i}-1)-2\sum_{k=0}^{m-1}\prod_{k+1}^{m-1}(2^{2^i}-1)$ and so the statement follows from Lemma \ref{lem:translationset}.
\end{proof}

\section{Open problems}

One of the most challenging problems in the theory of linear sets concerns the classification of all possible point–weight spectra; see, for instance, \cite{de2022weight,bonoli2005fq}.
In this paper, we provided new insight into this direction by giving criteria to determine the set of points of a fixed weight for linear sets with complementary weights.
However, the determination of all possible sizes of these sets remains unclear, and several fundamental questions remain open.

\begin{open}
    Given $k,n$, determine all possible weight distributions of $\mathbb{F}_q$-linear sets
$L_{S\times T} \subseteq \mathrm{PG}(1,q^n)$ with complementary weights.
\end{open}

Corollary~\ref{weg} gives a geometric criterion for the existence of linear sets with exactly two points of weight at least~$2$.
Therefore, naturally the following problem arises.

\begin{open}
Provide a complete algebraic classification of all such pairs $(S,T)$ defining linear sets with exactly two points of weight at least~$2$.
\end{open}

Corollary~\ref{cor:weightscomplpoli} expresses point weights using kernels of certain linearized polynomials.  
A major open problem is to translate these expressions into explicit and computable
criteria for general families of polynomials (monomials, binomials, trinomials, etc.).

\begin{open}
For which polynomials does one obtain extremal or very sparse point-weight distributions?
\end{open}

Propositions~\ref{prop:tracewithf} and \ref{prop:fxf} provide general upper bounds on the weight of points.  

\begin{open}
Determine for which polynomials these bounds are tight.
\end{open}

\section*{Acknowledgments}

The second author is very grateful for the hospitality of the School of Mathematics and Statistics of University of Canterbury, he was visiting it during the development of this research in August 2023.
The first author was partially supported by the Italian National Group for Algebraic and Geometric Structures and their Applications (GNSAGA - INdAM) CUP E53C24001950001.
The last author was partially supported by the Italian National Group for Algebraic and Geometric Structures and their Applications (GNSAGA - INdAM).

\bibliographystyle{abbrv}

\end{document}